\documentclass[a4paper,10pt]{article}
\usepackage{amsmath,amsthm,xypic,
amssymb,amsfonts,
pdfsync,stmaryrd,mathrsfs,yfonts
}
\newtheorem{thm}{Theorem}[section]

\newtheorem{prop}[thm]{Proposition}
\newtheorem{cor}[thm]{Corollary}

\newtheorem{lem}[thm]{Lemma}
\theoremstyle{definition}

\newtheorem{defi}[thm]{Definition}

\newtheorem{rem}[thm]{Remark}

\newtheorem{notation}[thm]{Notation}
\newtheorem{exm}[thm]{Example}

\theoremstyle{plain}

\newcommand{\lla}{\langle\!\langle}
\newcommand{\rra}{\rangle\!\rangle}

\newcommand{\trd}{\mathop{\mathrm{Trd}}}

\renewcommand{\phi}{\varphi}

\newcommand{\alt}{\mathop{\mathrm{Alt}}}

\newcommand{\sym}{\mathop{\mathrm{Sym}}}

\newcommand{\ad}{\mathop{\mathrm{Ad}}}

\newcommand{\End}{\mathop{\mathrm{End}}}

\newcommand{\car}{\mathop{\mathrm{char}}}
\newcommand{\id}{\mathop{\mathrm{id}}}

\newcommand{\disc}{\mathop{\mathrm{disc}}}

\author{A.-H. Nokhodkar}
\date{}
\title
{
Symmetric square-central elements in products of orthogonal involutions
in characteristic two
}

\begin{document}
\maketitle

\begin{abstract}
We obtain some criteria for a symmetric square-central element of a totally decomposable algebra with orthogonal involution in characteristic two, to be contained in an invariant quaternion subalgebra.
\end{abstract}

\section{Introduction}
A classical question concerning central simple algebras is to identify under which conditions a square-central element lies in a quaternion subalgebra.
This question is only solved for certain special cases in the literature.
Let $A$ be a central simple algebra of exponent two over a field $F$.
In \cite[(3.2)]{barry2014} it was shown that if $\car F\neq2$ and $F$ is of cohomological dimension less than or equal $2$, then every square-central element of $A$ is contained in a quaternion subalgebra (see also \cite[(4.2)]{barry-chapman}).
In \cite[(4.1)]{barry}, it was shown that if $\car F\neq2$, then an element $x\in A$ with $x^2=\lambda^2\in F^{\times2}$ lies in a (split) quaternion subalgebra if and only if $\dim_F(x-\lambda)A=\frac{1}{2}\dim_F A$.
This result was generalized in \cite[(3.2)]{me2} to arbitrary characteristics, including $\lambda=0$.
On the other hand, in \cite{rowen78} it was shown that there is an indecomposable algebra of degree $8$ and exponent $2$, containing a square-central element (see \cite[(5.6.10)]{jacob}).
Using similar methods, it was shown in \cite{barry} that for $n\geqslant3$ there exists a tensor product of $n$ quaternion algebras containing a square-central element which does not lie in any quaternion subalgebra.

A similar question for a central simple algebra with involution $(A,\sigma)$ over $F$ is whether a symmetric or skew-symmetric square-central element of $A$ lies in a $\sigma$-invariant quaternion subalgebra.
In the case where $\car F\neq2$, $\deg_FA=8$, $\sigma$ has trivial discriminant, the index of $A$ is
not $2$ and one of the components of the Clifford algebra $C(A,\sigma)$ splits, it was shown in \cite[(3.14)]{tignol} that every skew-symmetric square-central element of $A$ lies in a $\sigma$-invariant quaternion subalgebra.
In \cite{me1}, some criteria were obtained for symmetric and skew-symmetric elements whose squares lie in $F^2$, to be contained in a $\sigma$-invariant quaternion subalgebra.
Also, a sufficient condition was obtained in \cite[(6.3)]{mn1} for symmetric square-central elements in a totally decomposable algebra with orthogonal involution in characteristic two, to be contained in a stable quaternion subalgebra.

In this work we study the aforementioned question for totally decomposable algebras with orthogonal involution in characteristic two.
Let $(A,\sigma)$ be a totally decomposable algebra with orthogonal involution over a field $F$ of characteristic two and let $x\in A\setminus F$ be a symmetric element with $\alpha:=x^2\in F$.
Since the case where $\alpha\in F^2$ was settled in \cite{me1}, we assume that $\alpha\in F^\times\setminus F^{\times2}$.
We first study in \S\ref{sec-ins} some properties of {\it inseparable subalgebras}, introduced in \cite{mn1}.
It is shown in Theorem \ref{exm2} that $(A,\sigma)$ has a unique inseparable subalgebra if and only if either $\deg_FA\leqslant4$ or $\sigma$ is anisotropic.
In \S\ref{sec-rep}, the notion of the {\it representation} of an element of $F$ by $\sigma$ is defined.
Theorem \ref{rep} shows that the set of elements in $F$ represented by $\sigma$ coincides with the set of elements in $F$ represented by the {\it Pfister invariant} of $(A,\sigma)$ (a bilinear Pfister form associated to $(A,\sigma)$ and defined in \cite{dolphin3}).
In particular, if $0\neq x\in A$ is symmetric and square-central, then the element $x^2\in F$ is represented by the Pfister invariant (see Corollary \ref{Q}).

We then study our main problem in \S\ref{sec-quat} and \S\ref{sec-exm}.
For the case where $\sigma$ is anisotropic or $A$ has degree $4$, it is shown that every symmetric square-central element of $A$ lies in a $\sigma$-invariant quaternion subalgebra (see Theorem \ref{cor} and Proposition \ref{deg4}).
However, we will see in Proposition \ref{count} that if $\sigma$ is isotropic, $\deg_FA\geqslant8$ and $(A,\sigma)\not\simeq(M_{2^n}(F),t)$, there always exists a symmetric square-central of $(A,\sigma)$ which is not contained in any $\sigma$-invariant quaternion subalgebra of $A$.
If $A$ has degree $8$ or $\sigma$ satisfies certain isotropy condition, it is shown in Proposition \ref{8} and Theorem \ref{main2} that a symmetric square-central element of $A$ lies in a $\sigma$-invariant quaternion subalgebra if and only if it is contained in an inseparable subalgebra of $(A,\sigma)$.
Finally, in Example \ref{exm1} we shall see that this criterion cannot be applied to arbitrary involutions.

\section{Preliminaries}
Throughout this paper, $F$ denote a field of characteristic two.

Let $A$ be a central simple algebra over $F$.
An {\it involution} on $A$ is an anti\-automorphism $\sigma:A\rightarrow A$ of order two.
If $\sigma|_F=\id$, we say that $\sigma$ is of {\it the first kind}.
The sets of {\it alternating} and {\it symmetric} elements of $(A,\sigma)$ are defined as
\[\sym(A,\sigma)=\{x\in A\mid \sigma(x)=x\}\quad {\rm and} \quad \alt(A,\sigma)=\{\sigma(x)-x\mid x\in A\}.\]
For a field extension $K/F$ we use the notation $A_K=A\otimes K$, $\sigma_K=\sigma\otimes \id$ and $(A,\sigma)_K=(A_K,\sigma_K)$.
An extension $K/F$ is called a {\it splitting field}  of $A$ if $A_K$ splits, i.e., $A_K$ is isomorphic to the matrix algebra $M_n(K)$, where $n=\deg_FA$ is the degree of $A$ over $F$.
If $(V,\mathfrak{b})$ is a symmetric bilinear space over $F$, the pair $(\End_F(V),\sigma_\mathfrak{b})$ is denoted by $\ad(\mathfrak{b})$, where $\sigma_\mathfrak{b}$ is the {\it adjoint involution} of $\End_F(V)$ with respect to $\mathfrak{b}$ (see \cite[p. 2]{knus}).
According to \cite[(2.1)]{knus}, if $K$ is a splitting field of $A$, then $(A,\sigma)_K$ is adjoint to a symmetric bilinear form.
We say that $\sigma$ is {\it symplectic}, if this form is alternating.
Otherwise, $\sigma$ is called {\it orthogonal}.
By \cite[(2.6)]{knus}, $\sigma$ is symplectic if and only $1\in\alt(A,\sigma)$.
If $\sigma$ is an orthogonal involution, the {\it discriminant} of $\sigma$ is denoted by $\disc\sigma$ (see \cite[(7.2)]{knus}).

Let $(V,\mathfrak{b})$ be a bilinear space over $F$ and let $\alpha\in F$.
We say that $\mathfrak{b}$ {\it represents} $\alpha$  if $\mathfrak{b}(v,v)=\alpha$ for some nonzero vector $v\in V$.
The set of elements in $F$ represented by $\mathfrak{b}$ is denoted by $D(\mathfrak{b})$.
We also set $Q(\mathfrak{b})=D(\mathfrak{b})\cup \{0\}$.
Observe that $Q(\mathfrak{b})$  is an $F^2$-subspace of $F$.
If $K/F$ is a field extension, the scalar extension of $\mathfrak{b}$ to $K$ is denoted by $\mathfrak{b}_K$.
For $\alpha_1,\cdots,\alpha_n\in F^\times$, the diagonal bilinear form
$\sum_{i=1}^n\alpha_ix_iy_i$ is denoted by $\langle \alpha_1,\cdots,\alpha_n\rangle$.
The form $\lla \alpha_1,\cdots,\alpha_n\rra:=\langle 1,\alpha_1\rangle\otimes\cdots\otimes\langle 1,\alpha_n\rangle$ is called a {\it bilinear ($n$-fold) Pfister form}.
If $\mathfrak{b}$ is a bilinear Pfister form over $F$, then there exists a bilinear form $\mathfrak{b}'$, called the {\it pure subform of $\mathfrak{b}$}, such that $\mathfrak{b}\simeq\langle1\rangle\perp\mathfrak{b}'$.
The form $\mathfrak{b}'$ is uniquely determined, up to isometry (see \cite[p. 906]{arason}).

\section{The inseparable subalgebra}\label{sec-ins}
\begin{defi}
An involution $\sigma$ on a central simple algebra $A$ is called {\it isotropic} if $\sigma(x)x=0$ for some nonzero element $x\in A$.
Otherwise, $\sigma$ is called {\it anisotropic}.
\end{defi}
The following definition was given in \cite[p. 1629]{dolphin2}.
\begin{defi}
A central simple algebra with involution $(A,\sigma)$ is called {\it direct} if $\sigma(x)x\in\alt(A,\sigma)$ for $x\in A$ implies that $x=0$.
\end{defi}
It is clear that every direct algebra with involution is anisotropic.
\begin{notation}
For an algebra with involution $(A,\sigma)$ over $F$ we use the following notation:
\begin{align*}
  \alt(A,\sigma)^+&=\{x\in\alt(A,\sigma)\mid x^2\in F\}, \\
  \sym(A,\sigma)^+&=\{x\in\sym(A,\sigma)\mid x^2\in F\}.
\end{align*}
Note that we have $\alt(A,\sigma)^+\oplus F\subseteq\sym(A,\sigma)^+$.
\end{notation}
\begin{prop}\label{direct}
If $(A,\sigma)$ is a direct algebra with orthogonal involution over $F$, then $\sym(A,\sigma)^+$ is a subfield of $A$.
\end{prop}

\begin{proof}
Let $x,y\in \sym(A,\sigma)^+$.
Set $\alpha=x^2\in F$ and $\beta=y^2\in F$.
Consider the element $w:=xy+yx\in\sym(A,\sigma)$.
We have
\begin{align*}
\sigma(w)w&=w^2=(xy+yx)^2=xyxy+yxyx+xy^2x+yx^2y\\
&=xyxy+yxyx+\alpha\beta+\alpha\beta
=\sigma(yxyx)-yxyx\in\alt(A,\sigma).
\end{align*}
Hence $w=0$, i.e., $xy=yx$.
It follows that $\sigma(xy)=xy$ and $(xy)^2\in F$, i.e., $xy\in\sym(A,\sigma)^+$.
We also have $(x+y)^2=\alpha+\beta+xy+yx=\alpha+\beta\in F$,  hence $x+y\in\sym(A,\sigma)^+$.
The set  $\sym(A,\sigma)^+$ is therefore a commutative subalgebra of $A$ with $u^2\in F$ for every $u\in\sym(A,\sigma)^+$.
Since $\sigma$ is direct, if $0\neq u\in\sym(A,\sigma)^+$, then $u^2\neq0$, i.e., $u$ is a unit.
Hence, $\sym(A,\sigma)^+$ is a field.
\end{proof}

A central simple algebra with involution is called {\it totally decomposable} if it decomposes into tensor products of quaternion algebras with involution.
Let $(A,\sigma)\simeq\bigotimes_{i=1}^n(Q_i,\sigma_i)$ be a totally decomposable algebra of degree $2^n$ with orthogonal involution over $F$.
By \cite[(2.23)]{knus} the involution $\sigma_i$ is necessarily orthogonal for $i=1,\cdots,n$.
Also, according to \cite[(4.6)]{mn1} there exists a $2^n$-dimensional subalgebra $S\subseteq\sym(A,\sigma)^+$, called an {\it inseparable subalgebra} of $(A,\sigma)$, satisfying (i) $C_A(S)=S$, where $C_A(S)$ is the centralizer of $S$ in $A$; and (ii) $S$ is generated, as an $F$-algebra, by $n$ elements.
Furthermore, for every inseparable subalgebra $S$ of $(A,\sigma)$ we have necessarily $S\subseteq \alt(A,\sigma)^+\oplus F$.
According to \cite[(5.10)]{mn1}, if $S_1$ and $S_2$ are two inseparable subalgebras of $(A,\sigma)$, then $S_1\simeq S_2$ as $F$-algebras.
Note that if $v_i\in\sym(Q_i,\sigma_i)^+\setminus F$ is a unit for $i=1,\cdots,n$, then $F[v_1,\cdots,v_n]$ is an inseparable subalgebra of $(A,\sigma)$.

\begin{thm}\label{main}
Let $(A,\sigma)$ be a totally decomposable algebra with anisotropic orthogonal involution over $F$ and let $S$ be an inseparable subalgebra of $(A,\sigma)$.
Then $S=\alt(A,\sigma)^+\oplus F=\sym(A,\sigma)^+$ is a maximal subfield of $A$.
In particular, the inseparable subalgebra $S$ is uniquely determined.
\end{thm}

\begin{proof}
We already know that $S\subseteq \alt(A,\sigma)^+\oplus F\subseteq\sym(A,\sigma)^+$.
By \cite[(6.1)]{dolphin3}, $(A,\sigma)$ is direct.
Hence, $\sym(A,\sigma)^+$ is a field by Proposition \ref{direct}.
Since $S$ is maximal commutative, we obtain $\sym(A,\sigma)^+=S$.
Finally, as $\dim_F S=\deg_FA$, $S$ is a maximal subfield of $A$.
\end{proof}

The following definition was given in \cite{dolphin3}.
\begin{defi}
Let $(A,\sigma)\simeq\bigotimes_{i=1}^n(Q_i,\sigma_i)$ be a totally decomposable algebra with orthogonal involution over $F$.
The {\it Pfister invariant} of $(A,\sigma)$ is defined as $\mathfrak{Pf}(A,\sigma):=\lla\alpha_1,\cdots,\alpha_n\rra$,
where $\alpha_i\in F^\times$ is a representative of the class $\disc\sigma_i\in F^\times/F^{\times2}$, $i=1,\cdots,n$.
\end{defi}
According to \cite[(7.2)]{dolphin3}, the isometry class of the Pfister invariant is independent of the decomposition of $(A,\sigma)$.
Moreover, if $S$ is an inseparable subalgebra of $(A,\sigma)$, then $S$ may be considered as an underlying vector space of $\mathfrak{Pf}(A,\sigma)$ such that $\mathfrak{Pf}(A,\sigma)(x,x)=x^2$ for $x\in S$ (see \cite[(5.5)]{mn1}).

\begin{notation}
For a positive integer $n$, we denote the bilinear $n$-fold Pfister form $\lla1,\cdots,1\rra$ by $\lla1\rra^n$.
We also set $\lla1\rra^0=\langle1\rangle$.
\end{notation}

Let $\mathfrak{b}$ be a bilinear Pfister form over $F$.
In view of \cite[A.5]{arason}, one can find a non-negative integer $r$ and an anisotropic bilinear Pfister form $\mathfrak{b}'$ such that $\mathfrak{b}\simeq\lla1\rra^r\otimes\mathfrak{b}'$.
As in \cite{me1}, we denote the integer $r$ by $\mathfrak{i}(\mathfrak{b})$.
If $(A,\sigma)$ is a totally decomposable $F$-algebra with orthogonal involution, we simply denote $\mathfrak{i}(\mathfrak{Pf}(A,\sigma))$ by $\mathfrak{i}(A,\sigma)$.
By \cite[(5.7)]{dolphin3}, $(A,\sigma)$ is anisotropic if and only if $\mathfrak{i}(A,\sigma)=0$.
Hence, if $\sigma$ is isotropic,  then the pure subform of $\mathfrak{Pf}(A,\sigma)$ represents $1$.
In other words, if $S$ is an inseparable subalgebra of $(A,\sigma)$, then there exists $x\in S\setminus F$ with $x^2=1$.
Also, if $r=\mathfrak{i}(A,\sigma)>0$, there exists a totally decomposable algebra with anisotropic orthogonal involution $(B,\rho)$ over $F$ such that $(A,\sigma)\simeq(M_{2^r}(F),t)\otimes(B,\rho)$, where $t$ is the transpose involution (see \cite[p. 7]{me1}).
In particular, if $A$ is of degree $2^n$ then $\mathfrak{i}(A,\sigma)=n$ if and only if $(A,\sigma)\simeq(M_{2^n}(F),t)$.

We recall that every quaternion algebra $Q$ over $F$ has a {\it quaternion basis}, i.e., a basis $(1,u,v,w)$ satisfying $u^2+u\in F$, $v^2\in F^\times$ and $w=uv=vu+v$ (see \cite[p. 25]{knus}).
In this case, $Q$ is denoted by $[\alpha,\beta)_F$, where $\alpha=u^2+u\in F$ and $\beta=v^2\in F^\times$.

\begin{lem}\label{bas}
If $(Q,\sigma)$ is a quaternion algebra with orthogonal involution over $F$, then there is a quaternion basis $(1,u,v,w)$ of $Q$ such that $u,v\in\sym(Q,\sigma)$.
\end{lem}

\begin{proof}
Let $v\in\alt(Q,\sigma)$ be a unit.
Since $v\notin F$ and  $v^2\in F^\times$, it is easily seen that $v$ extends to a quaternion basis $(1,u,v,w)$ of $Q$.
By \cite[(4.5)]{me1}, we have $\sigma(u)=u$.
\end{proof}

\begin{lem}\label{3}
Let $(A,\sigma)$ be a totally decomposable algebra of degree $8$ with ortho\-gonal involution over $F$.
If $\sigma$ is isotropic, then there are two inseparable subalgebras $S_1$ and $S_2$ of $(A,\sigma)$ with $S_1\neq S_2$.
\end{lem}

\begin{proof}
Since $\mathfrak{i}(A,\sigma)>0$, we may identify $(A,\sigma)=(B,\rho)\otimes(M_2(F),t)$, where
$(B,\rho)\simeq(Q_1,\sigma_1)\otimes(Q_2,\sigma_2)$ is a tensor product of two quaternion algebras with orthogonal involution.
By Lemma \ref{bas} there exists a quaternion basis $(1,u_i,v_i,w_i)$ of $Q_i$ over $F$ such that $u_i,v_i\in\sym(Q_i,\sigma_i)$, $i=1,2$.
Let $v_3\in\alt(M_2(F),t)$ be a unit.
By scaling we may assume that $v_3^2=1$, because $\disc t$ is trivial (see \cite[p. 82]{knus}).
Then $S_1=F[v_1\otimes1,v_2\otimes1,1\otimes v_3]$ is an inseparable subalgebra of $(A,\sigma)$ (we have identified $Q_1$ and $Q_2$ with subalgebras of $B$).
Let
\begin{align*}
v'_1&=u_2v_1\otimes1+(u_2v_1+v_1)\otimes v_3\in \sym(A,\sigma)^+,\\
v'_2&=u_1v_2\otimes1+(u_1v_2+v_2)\otimes v_3\in \sym(A,\sigma)^+.
\end{align*}
Computation shows that $S_2=F[v'_1,v'_2,1\otimes v_3]$ is another inseparable subalgebra of $(A,\sigma)$.
Clearly, we have $S_1\neq S_2$.
\end{proof}

\begin{thm}\label{exm2}
A totally decomposable algebra with orthogonal involution $(A,\sigma)$ over $F$ has a unique inseparable subalgebra if and only if either $\deg_FA\leqslant 4$ or $\sigma$ is anisotropic.
\end{thm}

\begin{proof}
Let $S$ be an inseparable subalgebra of $(A,\sigma)$.
If $A$ is a quaternion algebra, then $S=\alt(A,\sigma)\oplus F$ by dimension count.
If $\deg_FA=4$, then $S=\alt(A,\sigma)^+\oplus F$ by \cite[(4.4)]{me3}.
Also, if $\sigma$ is anisotropic, then $S$ is uniquely determined by Theorem \ref{main}.
This proves the `if' implication.
To prove the converse, let $\deg_FA=2^n$.
 Suppose that $\sigma$ is isotropic and $\deg_FA\geqslant8$, i.e., $n\geqslant3$.
As $\mathfrak{i}(A,\sigma)>0$, we may identify $(A,\sigma)=\bigotimes_{i=1}^{n-1}(Q_i,\sigma_i)\otimes(M_2(F),t)$, where $(Q_i,\sigma_i)$ is a quaternion algebra with orthogonal involution over $F$ for $i=1,\cdots,n-1$.
By Lemma \ref{3} the algebra with involution
\[(Q_{n-2},\sigma_{n-2})\otimes(Q_{n-1},\sigma_{n-1})\otimes(M_2(F),t),\]
 admits two inseparable subalgebras $S_1$ and $S_2$ with $S_1\neq S_2$.
Let $S$ be an inseparable subalgebra of $\bigotimes_{i=1}^{n-3}(Q_i,\sigma_i)$.
Then $S\otimes S_1$ and $S\otimes S_2$ are two inseparable subalgebras of $(A,\sigma)$ with $S\otimes S_1\neq S\otimes S_2$, proving the result.
\end{proof}

\begin{lem}\label{pos}
Let $(A,\sigma)$ be a totally decomposable algebra of degree $2^n$ with orthogonal involution over $F$ and let $x\in\sym(A,\sigma)^+$ be a unit.
If $S$ is an inseparable subalgebra of $(A,\sigma)$, then for every unit $y\in S$, there exists a positive integer $k$ such that $(xy)^k\in \sym(A,\sigma)^+$.
In addition, for such an integer $k$ we have $(xy)^kx=x(xy)^k$.
\end{lem}

\begin{proof}
Since $x$ and $y$ are units, the element $(xy)^r$ is a unit for every integer $r$.
For $r\geqslant0$ let $S_r=(xy)^rS(xy)^{-r}$.
Then $S_r$ is a $2^n$-dimensional commutative subalgebra of $A$, which is generated by $n$ elements and satisfies $u^2\in F$ for every $u\in S_r$.
Set $\alpha=x^2\in F^\times$ and $\beta=y^2\in F^\times$.
Then
\[(xy)^{-r}=(y^{-1}x^{-1})^{r}=(\beta^{-1}y\alpha^{-1}x)^r=\alpha^{-r}\beta^{-r}(yx)^{r}.\]
Hence, $S_r=\alpha^{-r}\beta^{-r}(xy)^rS(yx)^r\subseteq\sym(A,\sigma)$, i.e., $S_r$ is an inseparable subalgebra of $(A,\sigma)$.
However, there exists a finite number of inseparable subalgebras of $(A,\sigma)$, so $S_r=S_s$ for some non-negative integers $r,s$ with $r>s$.
It follows that $S_{r-s}=S_0=S$.
In particular, we have $(xy)^{r-s}y(xy)^{s-r}\in S$ and
\begin{equation}\label{eq4}
(xy)^{r-s}y(xy)^{s-r}y=y(xy)^{r-s}y(xy)^{s-r}.
\end{equation}
Set $\lambda=\alpha^{s-r}\beta^{s-r}$, so that $(xy)^{s-r}=\lambda (yx)^{r-s}$.
Replacing in (\ref{eq4}) we obtain
\[\lambda(xy)^{r-s}y(yx)^{r-s}y=\lambda y(xy)^{r-s}y(yx)^{r-s}.\]
It follows that $\lambda y^2(xy)^{2(r-s)}=\lambda y^2(yx)^{2(r-s)}$, because $y^2\in F^\times$.
Hence, $(xy)^{k}=(yx)^{k}$, where $k=2(r-s)$.
We also have $\sigma((xy)^{k})=(yx)^{k}=(xy)^{k}$ and
$((xy)^{k})^2=(xy)^{k}(yx)^{k}\in F^\times$, hence $(xy)^{k}\in\sym(A,\sigma)^+$.
Finally, we have
$(xy)^{k}x=x(yx)^{k}=x(xy)^k$, completing the proof.
\end{proof}

\begin{prop}\label{met}
Let $(A,\sigma)$ be a totally decomposable algebra with orthogonal involution over $F$ and let $x\in\sym(A,\sigma)^+$ with $x^2\notin F^2$.
Then $\sigma$ is isotropic if and only if $\sigma|_{C_A(x)}$ is isotropic.
\end{prop}

\begin{proof}
Observe first that since $x^2\notin F^2$, $C_A(x)$ is a central simple algebra over $F(x)=F(\sqrt\alpha)$, where $\alpha=x^2\in F^\times$.
If $\sigma|_{C_A(x)}$ is isotropic, then $\sigma$ is clearly isotropic.
To prove the converse, let $S$ be an inseparable subalgebra of $(A,\sigma)$.
Since $\sigma$ is isotopic, there exists $y\in S\setminus F$ with $y^2=1$.
By Lemma \ref{pos}, there is a positive integer $k$ such that $(xy)^k\in \sym(A,\sigma)^+$.
Let $r$ be the minimum positive integer with $(xy)^r\in \sym(A,\sigma)^+$, hence $(xy)^r=(yx)^r$.
We claim that $(xy)^r\neq x^r$.
Suppose that $(xy)^r=x^r$.
If $r$ is odd, write $r=2s+1$ for some non-negative integer $s$.
The equality $(xy)^r=x^r$ then implies that $(yx)^sy(xy)^s=x^{2s}=\alpha^s$.
As $(xy)^s=\alpha^s(yx)^{-s}$, we get $\alpha^{s}(yx)^sy(yx)^{-s}=\alpha^s$.
Hence, $y=1\in F$, which contradicts the assumption.
If $r$ is even, write $r=2s$ for some positive integer $s$, so that $(xy)^r=x^r=\alpha^s$.
Multiplying with $(xy)^{-s}$ we obtain
\begin{align*}
(xy)^s=\alpha^s(xy)^{-s}=\alpha^s\alpha^{-s}(yx)^s=(yx)^s.
\end{align*}
It follows that $(xy)^s\in\sym(A,\sigma)^+$, contradicting the minimality of $r$.
This proves the claim.
According to Lemma \ref{pos}, we have $(xy)^r\in C_A(x)$.
Set $z=(xy)^r+x^r\in C_A(x)$.
Then $z\neq0$ and
$\sigma(z)z=\alpha^{r}+\alpha^{r}=0$, i.e., $\sigma|_{C_A(x)}$ is isotropic.
\end{proof}

\section{The set of elements represented by $\sigma$}\label{sec-rep}

\begin{defi}
Let $(A,\sigma)$ be a totally decomposable algebra with orthogonal involution over $F$ and let $\alpha\in F$.
We say that $\sigma$ {\it represents} $\alpha$ if there exists a nonzero element $x\in A$ such that $\sigma(x)x+\alpha\in\alt(A,\sigma)$.
The set of all elements in $F$ represented by $\sigma$ is denoted by $D(A,\sigma)$.
\end{defi}

The following result is a restatement of \cite[(6.1)]{dolphin3}.
\begin{lem}\label{dir}
Let $(A,\sigma)$ be a totally decomposable algebra with orthogonal invo\-lution over $F$.
Then $\sigma$ is anisotropic if and only if $0\notin D(A,\sigma)$.
\end{lem}

\begin{lem}\label{an}
Let $\mathfrak{b}$ be an anisotropic bilinear Pfister form over $F$ and let $\alpha\in F$.
If $\alpha\notin D(\mathfrak{b})$, then $\mathfrak{b}_{F(\sqrt{\alpha})}$ is anisotropic.
\end{lem}

\begin{proof}
Since $\alpha\notin D(\mathfrak{b})$, we have $\alpha\notin F^{\times2}$,
hence $F(\sqrt\alpha)$  is a field.
By \cite[(6.1)]{elman}, $\mathfrak{b}\otimes\lla\alpha\rra$ is anisotropic, hence $\mathfrak{b}_{F(\sqrt{\alpha})}$ is also anisotropic by \cite[(4.2)]{hof}.
\end{proof}

\begin{lem}\label{lam}
Let $(A,\sigma)$ be a totally decomposable algebra with orthogonal invo\-lution over $F$.
If $\sigma$ is anisotropic, then $D(A,\sigma)\subseteq D(\mathfrak{Pf}(A,\sigma))$.
\end{lem}

\begin{proof}
Let $\alpha\in D(A,\sigma)$.
Then there exists $0\neq x\in A$ such that $\sigma(x)x+\alpha\in\alt(A,\sigma)$.
Suppose that $\alpha\notin D(\mathfrak{Pf}(A,\sigma))$ and set $K=F(\sqrt{\alpha})$.
Since $\mathfrak{Pf}(A,\sigma)$ is a Pfister form, \cite[(7.5 (1))]{dolphin3} implies that $\mathfrak{Pf}(A,\sigma)$ is anisotropic.
Hence $\mathfrak{Pf}(A,\sigma)_K$ is also anisotropic by Lemma \ref{an}.
By \cite[(7.5 (2))]{dolphin3}, $(A,\sigma)_K$ is anisotropic.
As $\alt((A,\sigma)_K)=\alt(A,\sigma)\otimes K$, we have
$\sigma_K(x\otimes1)\cdot(x\otimes1)+\alpha\otimes1\in\alt((A,\sigma)_K)$.
Set $y=x\otimes1+1\otimes\sqrt{\alpha}\in A_K$.
Then $y\neq0$ and
\begin{align*}
  \sigma_K(y)y&=\sigma_K(x\otimes1+1\otimes\sqrt{\alpha})\cdot(x\otimes1+1\otimes\sqrt{\alpha})\\
  &=\sigma_K(x\otimes1)\cdot(x\otimes1)+\alpha\otimes1+(1\otimes\sqrt{\alpha})(x\otimes1+\sigma_K(x\otimes1)).
\end{align*}
We therefore obtain $\sigma_K(y)y\in\alt((A,\sigma)_K)$, i.e., $0\in D((A,\sigma)_K)$.
This contradicts Lemma \ref{dir}, because $(A,\sigma)_K$ is anisotropic.
\end{proof}

\begin{cor}\label{anlam}
Let $(A,\sigma)$ be a totally decomposable algebra with anisotropic orthogonal involution over $F$.
If
$\alpha+\sum_{i=1}^m \sigma(x_i)x_i\in\alt(A,\sigma)$ for some $\alpha\in F$ and $x_1,\cdots,x_m\in A$, then $\alpha\in Q(\mathfrak{Pf}(A,\sigma))$.
\end{cor}

\begin{proof}
Set $y=\sum_{i=1}^mx_i\in A$.
Then
\begin{align*}
  \alpha+\sigma(y)y &=\textstyle\alpha+\sum\limits _{i=1}^m\sigma(x_i)\cdot \sum\limits_{i=1}^mx_i=\alpha+\sum\limits_{i=1}^m\sigma(x_i)x_i+\sum\limits_{i\neq j}\sigma(x_j)x_i\\
  &=\textstyle(\alpha+\sum\limits_{i=1}^m\sigma(x_i)x_i)+\sum\limits_{i< j}(\sigma(x_j)x_i+\sigma(x_i)x_j)\\
  &=\textstyle(\alpha+\sum\limits_{i=1}^m\sigma(x_i)x_i)+\sum\limits_{i< j}(\sigma(\sigma(x_i)x_j)-\sigma(x_i)x_j).
\end{align*}
Hence, $\alpha+\sigma(y)y\in\alt(A,\sigma)$.
If $y=0$, then $\alpha\in\alt(A,\sigma)$ and the orthogonality of $\sigma$ implies that $\alpha=0\in Q(\mathfrak{Pf}(A,\sigma))$.
Otherwise, we have $\alpha\in D(A,\sigma)$ and the result follows from Lemma \ref{lam}.
\end{proof}

\begin{rem}\label{rem}
Let $(B,\rho)$ be a central simple algebra with involution over $F$ and set $(A,\sigma)=(B,\rho)\otimes(M_2(F),t)$.
Then every element $x\in A$ can be written as $\left(\begin{smallmatrix}a & b\\c & d\end{smallmatrix}\right)$, where $a,b,c,d\in B$.
The involution $\sigma$ maps $x$ to $\left(\begin{smallmatrix}
                                             \rho(a) & \rho(c) \\ \rho(b) & \rho(d)
                                            \end{smallmatrix}\right)$.
It follows that
\begin{align*}
\alt(A,\sigma)&=\{\left(\begin{matrix}
  a & b \\
  \rho(b) & c
\end{matrix}\right)\mid a,c\in\alt(B,\rho)\ \  {\rm and} \ \ b\in B\},\\
\sym(A,\sigma)&=\{\left(\begin{matrix}a & b \\ \rho(b) & c\end{matrix}\right)\mid a,c\in\sym(B,\rho)\ \  {\rm and} \ \ b\in B\}.
\end{align*}
\end{rem}

\begin{lem}\label{1}
If $\alpha I+\sum_{i=1}^mA_i^tA_i\in\alt(M_2(F),t)$, where $A_1,\cdots,A_m\in M_2(F)$ and $I\in M_2(F)$ is the identity matrix, then $\alpha\in F^2$.
\end{lem}

\begin{proof}
Write
$A_i=\left(\begin{smallmatrix}a_i & b_i \\c_i & d_i\end{smallmatrix}\right)$
for some $a_i,b_i,c_i,d_i\in F$, so that
\begin{align*}
A_i^tA_i=\left(\begin{matrix}a_i^2+c_i^2 & a_ib_i+c_id_i \\a_ib_i+c_id_i & b_i^2+d_i^2 \end{matrix}\right).
\end{align*}
Since $\alt(M_2(F),t)$ is the set of symmetric matrices with zero diagonal, the assumption implies that $\alpha+\sum_{i=1}^ma_i^2+c_i^2=0$, hence $\alpha\in F^2$.
\end{proof}

\begin{lem}\label{prop}
Let $(A,\sigma)$ be a totally decomposable algebra with orthogonal involution over $F$.
If $\alpha+\sum_{i=1}^m\sigma(x_i)x_i\in\alt(A,\sigma)$ for some $\alpha\in F$ and  $x_1,\cdots,x_m\in A$, then $\alpha\in Q(\mathfrak{Pf}(A,\sigma))$.
\end{lem}

\begin{proof}
Let $\deg_FA=2^n$.
We use induction on $n$.
If $\sigma$ is anisotropic, the conclusion follows from Corollary \ref{anlam}.
Suppose that $\sigma$ is isotropic.
If $n=1$, then $(A,\sigma)\simeq(M_2(F),t)$ and $\mathfrak{Pf}(A,\sigma)\simeq\lla1\rra$ and the result follows from Lemma \ref{1}.
Let $n\geqslant2$.
Since $\mathfrak{i}(A,\sigma)>0$, we may identify $(A,\sigma)=(B,\rho)\otimes(M_2(F),t)$, where $(B,\rho)$ is a totally decomposable algebra with orthogonal involution over $F$.
Then $\mathfrak{Pf}(A,\sigma)\simeq\mathfrak{Pf}(B,\rho)\otimes\lla1\rra$, which implies that $Q(\mathfrak{Pf}(A,\sigma))=Q(\mathfrak{Pf}(B,\rho))$.
Write
$x_i=\left(\begin{smallmatrix}a_i & b_i \\c_i & d_i\end{smallmatrix}\right)$,
where $a_i,b_i,c_i,d_i\in B$.
Using Remark \ref{rem} we get
\begin{align*}
\sigma(x_i)x_i=\left(\begin{matrix}\rho(a_i) & \rho(c_i) \\\rho(b_i) & \rho(d_i)\end{matrix}\right)\cdot\left(\begin{matrix}a_i & b_i \\c_i & d_i\end{matrix}\right).
\end{align*}
Hence the $(1,1)$-entry of $\sum_{i=1}^m\sigma(x_i)x_i$ is $\sum_{i=1}^m(\rho(a_i)a_i+\rho(c_i)c_i)$.
Since $\alpha+\sum_{i=1}^m\sigma(x_i)x_i\in \alt(A,\sigma)$, Remark \ref{rem} implies that \[\textstyle\alpha+\sum\limits_{i=1}^m(\rho(a_i)a_i+\rho(c_i)c_i)\in\alt(B,\rho).\]
By induction hypothesis we get $\alpha\in Q(\mathfrak{Pf}(B,\rho))=Q(\mathfrak{Pf}(A,\sigma))$.
\end{proof}

\begin{cor}\label{Q}
Let $(A,\sigma)$ be a totally decomposable algebra with orthogonal involution over $F$.
If $x\in\sym(A,\sigma)^+$, then $x^2\in Q(\mathfrak{Pf}(A,\sigma))$.
\end{cor}

\begin{proof}
Set $\alpha=x^2\in F$.
Then $\alpha+\sigma(x)x=0\in\alt(A,\sigma)$ and the result follows from Lemma \ref{prop}.
\end{proof}

\begin{thm}\label{rep}
If $(A,\sigma)$ is a totally decomposable algebra with orthogonal invo\-lution over $F$, then $D(A,\sigma)=D(\mathfrak{Pf}(A,\sigma))=\{x^2\mid 0\neq x\in S\}$, where $S$ is any inseparable subalgebra of $(A,\sigma)$.
\end{thm}

\begin{proof}
If $\alpha\in D(\mathfrak{Pf}(A,\sigma))$, then there exists $0\neq x\in S$ such that $x^2=\mathfrak{Pf}(A,\sigma)(x,x)=\alpha$. Hence, $\sigma(x)x+\alpha=0\in\alt(A,\sigma)$, i.e., $\alpha\in D(A,\sigma)$.
Conversely, let $\alpha\in D(A,\sigma)$.
If $\alpha=0$, then $\sigma$ is isotropic by Lemma \ref{dir}.
Hence, $\mathfrak{Pf}(A,\sigma)$ is isotopic by \cite[(7.5) and (6.2)]{dolphin3}, i.e, $0\in D(\mathfrak{Pf}(A,\sigma))$.
Otherwise, as $Q(\mathfrak{Pf}(A,\sigma))=D(\mathfrak{Pf}(A,\sigma))\cup\{0\}$, the result follows from Lemma \ref{prop}.
The equality $D(\mathfrak{Pf}(A,\sigma))=\{x^2\mid 0\neq x\in S\}$ follows from the relation $\mathfrak{Pf}(A,\sigma)(x,x)=x^2$ for $x\in S$.
\end{proof}

\begin{lem}\label{ib}
Let $\mathfrak{b}$ be a bilinear $n$-fold Pfister form over $F$.
If $\alpha\in Q(\mathfrak{b})\setminus F^2$, then $\mathfrak{i}(\mathfrak{b}_{F(\sqrt\alpha)})=\mathfrak{i}(\mathfrak{b})+1$.
\end{lem}

\begin{proof}
Set $K=F(\sqrt\alpha)$ and $r=\mathfrak{i}(\mathfrak{b})$.
As $Q(\lla1\rra)=F^2$ and $\alpha\in Q(\mathfrak{b})\setminus F^2$ we have $\mathfrak{b}\not\simeq\lla1\rra^n$, i.e., $r<n$.
Write $\mathfrak{b}\simeq\lla1\rra^r\otimes\mathfrak{b}'$ for some anisotropic bilinear Pfister form $\mathfrak{b}'$ over $F$.
Since $Q(\mathfrak{b})=Q(\mathfrak{b}')$, we have $\alpha\in Q(\mathfrak{b}')$.
Hence, the pure subform of $\mathfrak{b}'$ represents $\alpha+\lambda^2$ for some $\lambda\in F$.
By \cite[A.2]{arason}, there exist $\alpha_2,\cdots,\alpha_s\in F$ such that $\mathfrak{b}'\simeq\lla \alpha+\lambda^2,\alpha_2,\cdots,\alpha_s\rra$.
Note that we have $\alpha+\lambda^2\in K^{\times2}$, hence $\mathfrak{b}'_K\simeq\lla1,\alpha_2,\cdots,\alpha_s\rra_K$.
Since $\mathfrak{b}'=\lla \alpha+\lambda^2\rra\otimes\lla\alpha_2,\cdots,\alpha_s\rra$ is anisotropic and $K=F(\sqrt{\alpha+\lambda^2})$, by \cite[(4.2)]{hof} the form $\lla\alpha_2,\cdots,\alpha_s\rra_K$ is anisotropic.
Hence $\mathfrak{i}(\mathfrak{b}'_K)=1$, which implies that $\mathfrak{i}(\mathfrak{b}_{K})=r+1=\mathfrak{i}(\mathfrak{b})+1$.
\end{proof}

\begin{cor}\label{ia}
Let $(A,\sigma)$ be a totally decomposable algebra with orthogonal invo\-lution over $F$.
If $x\in\sym(A,\sigma)^+$ with $\alpha:=x^2\notin F^{\times2}$, then $\mathfrak{i}((A,\sigma)_{F(\sqrt\alpha)})=\mathfrak{i}(A,\sigma)+1$.
In particular, $(A,\sigma)\not\simeq(M_{2^n}(F),t)$.
\end{cor}

\begin{proof}
By Corollary \ref{Q}, we have $\alpha\in Q(\mathfrak{Pf}(A,\sigma))$.
Hence, the result follows from Lemma \ref{ib}.
\end{proof}

\section{Stable quaternion subalgebras}\label{sec-quat}
In this section we study some conditions under which a symmetric square-central element of a totally decomposable algebra with orthogonal involution is contained in a stable quaternion subalgebra.
We start with anisotropic involutions.

\begin{thm}\label{cor}
Let $(A,\sigma)$ be a totally decomposable algebra with anisotropic orthogonal involution over $F$.
Then every $x\in\sym(A,\sigma)^+$ is contained in a $\sigma$-invariant quaternion subalgebra of $A$.
\end{thm}

\begin{proof}
Since $\sigma$ is anisotropic, Theorem \ref{main} shows that $x$ is contained in the unique inseparable subalgebra of $(A,\sigma)$.
If $x^2=\lambda^2$ for some $\lambda\in F$, then $(x+\lambda)^2=0$.
Hence, $x=\lambda$ by \cite[(6.1)]{dolphin3} and the result is trivial.
Otherwise, $x^2\notin F^2$ and the conclusion follows from \cite[(6.3 (ii))]{mn1}.
\end{proof}

We next consider algebras of degree $4$ and $8$.
\begin{prop}\label{deg4}
Let $(A,\sigma)$ be a totally decomposable algebra of degree $4$ with orthogonal involution over $F$.
If $x\in\sym(A,\sigma)^+$ with $x^2\notin F^2$, then $x$ is contained in a $\sigma$-invariant quaternion subalgebra of $A$.
\end{prop}

\begin{proof}
By Corollary \ref{ia} we have either $\mathfrak{i}(A,\sigma)=0$ or $\mathfrak{i}(A,\sigma)=1$.
In the first case, the result follows from Theorem \ref{cor}.
Suppose that $\mathfrak{i}(A,\sigma)=1$.
Set $C=C_A(x)$ and $K=F(x)=F(\sqrt\alpha)$.
By Proposition \ref{met}, $(C,\sigma|_C)$ is isotropic, i.e., $(C,\sigma|_C)\simeq(M_2(K),t)\simeq(M_2(F),t)_K$.
Hence, the algebra $M_2(F)$ may be identified with a subalgebra of $C\subseteq A$.
Let $Q=C_A(M_2(F))$.
Then $(Q,\sigma|_Q)$ is a $\sigma$-invariant quaternion subalgebra of $A$ containing $x$.
\end{proof}

\begin{lem}\label{2}
Let $\mathfrak{b}$ be a $2$-dimensional anisotropic symmetric bilinear form over $F$ and let $K=F(\sqrt\alpha)$ for some $\alpha\in F^\times\setminus F^{\times2}$.
If $\mathfrak{b}_K$ is isotropic, then $\mathfrak{b}\simeq\langle a,sa\rangle$ for some $a\in F^\times$ and $s\in K^{\times2}$.
\end{lem}

\begin{proof}
The result follows from \cite[(3.5)]{lagh}.
\end{proof}

We recall that a symmetric bilinear space $(V,\mathfrak{b})$ over $F$ is called {\it metabolic} if there exists a subspace $W$ of $V$ with $\dim_FW=\frac{1}{2}\dim_FV$ such that $\mathfrak{b}|_{W\times W}=0$.
\begin{lem}\label{pfi}
Let $\mathfrak{b}$ be a $4$-dimensional isotropic symmetric bilinear form over $F$.
Let $K=F(\sqrt\alpha)$ for some $\alpha\in F^\times\setminus F^{\times2}$.
If $\mathfrak{b}\otimes\lla\alpha\rra$ is a Pfister form, then $\mathfrak{b}_K$ is similar to a Pfister form.
\end{lem}

\begin{proof}
Since $\mathfrak{b}\otimes\lla\alpha\rra$ is a Pfister form, the form $\mathfrak{b}$ is not alternating.
If $\mathfrak{b}$ is metabolic, then by \cite[(1.24) and (1.22 (3))]{elman} either $\mathfrak{b}\simeq\langle a,a,b,b\rangle$ or $\mathfrak{b}\simeq\langle a,a\rangle\perp\mathbb{H}$, where $a,b\in F^\times$ and $\mathbb{H}$ is the hyperbolic plane.
In the first case $\mathfrak{b}$ is similar to $\langle1,1,ab,ab\rangle=\lla1,ab\rra$.
In the second case, using the isometry $\langle a,a,a\rangle\simeq\langle a\rangle\perp\mathbb{H}$ in \cite[(1.16)]{elman}, we get $\mathfrak{b}\simeq\langle a,a,a,a\rangle$.
Hence, $\mathfrak{b}$ is similar to $\lla1,1\rra$.

Suppose that $\mathfrak{b}$ is not metabolic.
By Witt decomposition theorem \cite[(1.27)]{elman} and the isometry $\langle a,a,a\rangle\simeq\langle a\rangle\perp\mathbb{H}$ for $a\in F^\times$, the form $\mathfrak{b}$ is similar to $\mathfrak{b}'\perp\langle1,1\rangle$ for some $2$-dimensional anisotropic bilinear form $\mathfrak{b}'$ over $F$.
Since $\mathfrak{b}\otimes\langle1,\alpha\rangle$ is an isotropic Pfister form, it is metabolic by \cite[(6.3)]{elman}.
Hence, $\mathfrak{b}'\otimes\langle1,\alpha\rangle$ is also metabolic.
By \cite[(4.2)]{hof} the form $\mathfrak{b}'_K$ is metabolic.
Using Lemma \ref{2}, one can write $\mathfrak{b}'\simeq\langle a,sa\rangle$, where $a\in F^\times$ and $s\in K^{\times2}$.
The form $\mathfrak{b}_K$ is therefore similar to $\langle1,1,a,sa\rangle_K\simeq\langle1,1,a,a\rangle_K\simeq\lla1,a\rra_K$, proving the result
\end{proof}

\begin{lem}\label{4}
Let $(A,\sigma)$ be a central simple algebra of degree $4$ with orthogonal involution over $F$ and let $K/F$ be a separable quadratic extension.
If $(A,\sigma)_K$ is totally decomposable, then $(A,\sigma)$ is also totally decomposable.
\end{lem}

\begin{proof}
The result follows from \cite[(7.3)]{pari} and the relation $K^{\times2}\cap F^\times=F^{\times2}$.
\end{proof}

\begin{lem}\label{lem}
Let $(A,\sigma)$ be a totally decomposable algebra of degree $2^n$ with orthogonal involution over $F$.
Let $x\in\sym(A,\sigma)^+$ with $x^2\notin F^2$ and set $C=C_A(x)$.
If $(C,\sigma|_C)$ is totally decomposable and $x$ is contained in a $\sigma$-invariant quaternion subalgebra $Q$,
then there exists an inseparable subalgebra of $(A,\sigma)$ containing $x$.
\end{lem}

\begin{proof}
Let $\alpha=x^2\in F^\times\setminus F^{\times2}$, $K=F(x)=F(\sqrt\alpha)$, $B=C_A(Q)$ and $\rho=\sigma|_B$.
Then there exists a natural isomorphism of $K$-algebras with involution
\[f:(B,\rho)_K\simeq(C,\sigma|_C).\]
In particular, $(B,\rho)_K$ is totally decomposable.
By \cite[(3.6)]{me1}, there exists a totally decomposable algebra of degree $2^{n-1}$ with orthogonal involution $(B',\rho')$ over $F$ such that $(B,\rho)_K\simeq(B',\rho')_K$.
Let $S=F[v_1,\cdots,v_{n-1}]$ be an inseparable subalgebra of $(B',\rho')$.
The algebra $K[v_1,\cdots,v_{n-1}]\simeq S_K$ may be identified with an inseparable subalgebra of $(B,\rho)_K$.
Set $v'_i=f(v_i)\in C$, $i=1,\cdots,n-1$.
Then $S':=K[v'_1,\cdots,v'_{n-1}]$ is an inseparable subalgebra of $(C,\sigma|_C)$ satisfying $u^2\in F$ for $u\in S'$.
Since $(C,\sigma|_C)\subseteq (A,\sigma)$, we have $S'\subseteq\sym(A,\sigma)^+$.
By \cite[(3.11)]{mn1}, $S'$ is a Frobenius subalgebra of $A$, because $\dim_FS'=2^n$ and $S'$ is generated, as an $F$-algebra, by $x,v'_1,\cdots,v'_{n-1}$.
It follows from \cite[(2.2.3)]{jacob} that $C_A(S')=S'$.
Hence,
$S'$ is an inseparable subalgebra of $(A,\sigma)$ containing $x$.
\end{proof}

\begin{prop}\label{8}
Let $(A,\sigma)$ be a totally decomposable algebra of degree $8$ over $F$.
For an element $x\in\sym(A,\sigma)^+$ with $x^2\notin F^2$, the following conditions are equivalent:
\begin{itemize}
\item[(1)] There exists a $\sigma$-invariant quaternion subalgebra of $A$ containing $x$.
\item[(2)] There exists an inseparable subalgebra $S$ of $(A,\sigma)$ such that $x\in S$.
\end{itemize}
\end{prop}

\begin{proof}
If $\mathfrak{i}(A,\sigma)=0$, by Theorem \ref{cor} and Theorem \ref{main} both conditions are satisfied.
Let $\mathfrak{i}(A,\sigma)>0$.
Then $(A,\sigma)\simeq(M_2(F),t)\otimes(Q_1,\sigma_1)\otimes(Q_2,\sigma_2)$, where $(Q_i,\sigma_i)$, $i=1,2$, is a quaternion algebra with orthogonal involution over $F$.
Suppose first that $x$ is contained in a $\sigma$-invariant quaternion subalgebra $Q_3$ of $A$.
Set $B=C_A(Q_3)$, $\sigma_3=\sigma|_{Q_3}$ and $\rho=\sigma|_B$, so that
$(A,\sigma)\simeq(Q_3,\sigma_3)\otimes(B,\rho)$.
It follows that
\begin{align}\label{eq2}
(Q_3,\sigma_3)\otimes(B,\rho)\simeq(M_2(F),t)\otimes(Q_1,\sigma_1)\otimes(Q_2,\sigma_2).
\end{align}

Let $C=C_A(x)$ and $K=F(x)=F(\sqrt\alpha)$, where $\alpha=x^2\in F^\times\setminus F^{\times2}$.
Then $(C,\sigma|_C)\simeq(B,\rho)_K$ as $K$-algebras.
By Proposition \ref{met}, $(C,\sigma|_C)$ is isotropic, hence $(B,\rho)_K$ is also isotropic.
We claim that $(B,\rho)_K$ is totally decomposable.
The result then follows from Lemma \ref{lem}.

By Lemma \ref{bas}, for $i=1,2,3$ there exists a quaternion basis $(1,u_i,v_i,w_i)$ of $Q_i$ such that $u_i\in\sym(Q_i,\sigma_i)$.
Let $\beta_i=u_i^2+u_i\in F$.
For $i=0,1,2,3$, define a field $L_i$ inductively as follows:
set $L_0=F$.
For $i\geqslant1$ set $L_i=L_{i-1}(u_i)$ if $\beta_i\notin\wp(L_{i-1}):=\{y^2+y\mid y\in L_{i-1}\}$ and $L_i=L_{i-1}$ otherwise.
In other words, either $L_i=L_{i-1}$ or $L_i/L_{i-1}$ is a separable quadratic extension.
Note that we have $L_i^{\times2}\cap F^\times=F^{\times2}$, hence either $L_i(\sqrt\alpha)=L_{i-1}(\sqrt\alpha)$  or $L_i(\sqrt\alpha)/L_{i-1}(\sqrt\alpha)$ is a separable quadratic extension.
We show that $\rho_{L_3(\sqrt\alpha)}$ is totally decomposable, which implies that $\rho_{L_i(\sqrt\alpha)}$ is totally decomposable for $i=0,1,2$, thanks to Lemma \ref{4}.
In particular, $\rho_K=\rho_{F(\sqrt\alpha)}$ is also totally decomposable, as required.

Set $L=L_3$.
For $i=1,2,3$, the algebra ${Q_i}_L$ splits.
Hence we may identify $(Q_i,\sigma_i)_{L}=(M_2(L),\tau_i)$, where $\tau_i$ is an orthogonal involution on $M_2(L)$.
Using (\ref{eq2}) we obtain
\begin{align}\label{eq6}
(M_2(L),\tau_3)\otimes(B,\rho)_{L}\simeq(M_2(L),t)\otimes(M_2(L),\tau_1)\otimes(M_2(L),\tau_2).
\end{align}
Hence, $B_L$ splits and we may identify $(B,\rho)_L=\ad(\mathfrak{b})$ for some symmetric bilinear form $\mathfrak{b}$ over $L$.
As $(B,\rho)_L$ is isotropic, the form $\mathfrak{b}$ is isotropic by \cite[(4.3)]{dolphin3}.
Since $x\in \sym(Q_3,\sigma_3)^+\setminus F$, by \cite[(5.4)]{me2} there exists $\lambda\in F$ for which $0\neq x+\lambda\in\alt(Q_3,\sigma_3)$.
Hence, $\disc\sigma_3=(\alpha+\lambda^2)F^{\times2}$, which implies that
\begin{equation}\label{eq3}
(M_2(L),\tau_3)\simeq(Q_3,\sigma_3)_L\simeq\ad(\lla\alpha+\lambda^2\rra_L),
\end{equation}
by \cite[(7.4)]{knus}.
The right side of (\ref{eq6}) is the adjoint involution of a $3$-fold bilinear Pfister form over $L$.
Hence, using (\ref{eq3}) and \cite[(4.2)]{knus} it follows that $\mathfrak{b}\otimes\lla\alpha+\lambda^2\rra$ is similar to a Pfister form over $L$.
By Lemma \ref{pfi}, $\mathfrak{b}_{L(\sqrt\alpha)}$ is similar to a Pfister form.
Hence, $\rho_{L(\sqrt\alpha)}$ is totally decomposable, as required.
This proves the implication $(1)\Rightarrow(2)$.
The converse follows from \cite[(6.3 (ii))]{mn1}.
\end{proof}

\begin{lem}\label{adj}
If $\mathfrak{b}$ is a $2^n$-dimensional symmetric bilinear form over $F$, then $\ad(\mathfrak{b})\simeq(M_{2^n}(F),t)$ if and only if $\mathfrak{b}\simeq\lambda\cdot\lla1\rra^n$ for some $\lambda\in F^\times$.
\end{lem}
\begin{proof}
The result follows from \cite[(5.7)]{mn1}, \cite[(7.5 (3))]{dolphin3} and \cite[(4.2)]{knus}.
\end{proof}

\begin{lem}\label{t}
Let $(A,\sigma)$ be a central simple algebra with orthogonal involution over $F$.
If $(A,\sigma)\otimes(M_2(F),\tau)\simeq(M_{2^n}(F),t)$ for some orthogonal involution $\tau$ on $M_2(F)$, then $(A,\sigma)\simeq(M_{2^{n-1}}(F),t)$.
\end{lem}

\begin{proof}
Observe first that $A$ splits, hence we may identify $(A,\sigma)=\ad(\mathfrak{b})$ for some symmetric bilinear form $\mathfrak{b}$ over $F$.
By \cite[(7.3 (3)) and (7.4)]{knus}, we have $(M_2(F),\tau)\simeq\ad(\lla\alpha\rra)$, where $\alpha\in F^\times$ is a representative of the class $\disc\tau\in F^\times/F^{\times2}$.
Also, $(M_{2^n}(F),t)\simeq\ad(\lla1\rra^n)$ by Lemma \ref{adj}.
The assumption implies that $\ad(\mathfrak{b}\otimes\lla\alpha\rra)\simeq\ad(\lla1\rra^n)$.
Hence, the forms $\mathfrak{b}\otimes\lla\alpha\rra$ and $\lla1\rra^n$ are similar by \cite[(4.2)]{knus}, i.e., there exists $\lambda\in F^\times$ for which $\mathfrak{b}\otimes\lla\alpha\rra\simeq\lambda\cdot\lla1\rra^n$.
As $Q(\lla1\rra^n)=F^2$, we obtain $Q(\mathfrak{b})\subseteq\lambda\cdot F^2$, i.e., $\mathfrak{b}$ is similar to $\lla1\rra^{n-1}$.
Using Lemma \ref{adj} we get $(A,\sigma)\simeq(M_{2^{n-1}}(F),t)$.
\end{proof}

\begin{thm}\label{main2}
Let $(A,\sigma)$ be a totally decomposable algebra of degree $2^n$ with orthogonal involution over $F$ and let $x\in\sym(A,\sigma)^+$ with $x^2\notin F^2$.
If $\mathfrak{i}(A,\sigma)=n-1$, then the following statements are equivalent:
\begin{itemize}
\item[(1)] There exists a $\sigma$-invariant quaternion subalgebra $Q$ of $A$ containing $x$.
\item[(2)] There exists an inseparable subalgebra $S$ of $(A,\sigma)$ such that $x\in S$.
\end{itemize}
\end{thm}

\begin{proof}
The implication $(2)\Rightarrow(1)$ follows from \cite[(6.3 (ii))]{mn1}.
To prove the converse, let $C=C_A(x)$.
In view of Lemma \ref{lem}, it suffices to show that $(C,\sigma|_C)$ is totally decomposable.
Let $\tau=\sigma|_Q$, $B=C_A(Q)$ and $\rho=\sigma|_B$.
Then $(A,\sigma)\simeq(B,\rho)\otimes(Q,\tau)$.
Set $K=F(x)=F(\sqrt\alpha)$, where $\alpha=x^2\in F^\times\setminus F^{\times2}$.
We have
 $(C,\sigma|_C)\simeq_K(B,\rho)_K$.
Hence, it is enough to show that $(B,\rho)_K$ is totally decomposable.
By Corollary \ref{ia} we have $\mathfrak{i}(A,\sigma)_K=n$, so $(A,\sigma)_K\simeq(M_{2^n}(K),t)$.
It follows that $(B,\rho)_K\otimes_K(Q,\tau)_K\simeq(M_{2^n}(K),t)$.
Since $x\in Q$ and $x^2=\alpha\in K^2$, the algebra $Q_K$ splits.
Hence, using Lemma \ref{t} we get
$(B,\rho)_K\simeq_K(M_{2^{n-1}}(K),t)$.
In particular, $(B,\rho)_K$ is totally decomposable, proving the result.
\end{proof}

\section{Examples for isotropic involutions}\label{sec-exm}
In this section we show that the criteria obtained in \S\ref{sec-quat} do not necessarily apply to arbitrary involutions.
\begin{lem}\label{sym}
Let $(A,\sigma)$ be a totally decomposable algebra of degree $2^n$ with orthogonal involution over $F$.
If  $n\geqslant2$ and $(A,\sigma)\not\simeq(M_{2^n}(F),t)$, then there exist an element $w\in\sym(A,\sigma)\setminus (\alt(A,\sigma)\oplus F)$ and a unit $u\in\alt(A,\sigma)$ such that $u^2\in F^\times\setminus F^{\times2}$
and $uw=wu$.
\end{lem}

\begin{proof}
Let $(A,\sigma)\simeq\bigotimes_{i=1}^n(Q_i,\sigma_i)$ be a decomposition of $(A,\sigma)$.
As $(A,\sigma)\not\simeq(M_{2^n}(F),t)$, (by re-indexing) we may assume that $(Q_1,\sigma_1)\not\simeq(M_2(F),t)$.
Let $u\in\alt(Q_1,\sigma_1)$ be a unit, so that $u^2\in F^\times$.
Note that $u^2\notin F^{\times2}$, since otherwise $Q_1$ splits and $\disc\sigma_1$ is trivial.
As $\disc t$ is also trivial (see \cite[p. 82]{knus}), we get $(Q_1,\sigma_1)\simeq(M_2(F),t)$ by \cite[(7.4)]{knus}, contradicting the assumption.
Hence, $u^2\in F^\times\setminus F^{\times2}$.
By \cite[(2.6)]{knus} we have $\dim_F\sym(Q_2,\sigma_2)=3$ and $\dim_F\alt(Q_2,\sigma_2)=1$.
Hence there exists an element $w\in\sym(Q_2,\sigma_2)\setminus (\alt(Q_2,\sigma_2)\oplus F)$.
Clearly, we have $uw=wu$.
The elements $u$ and $w$ may be identified with elements of $A$, so that $w\in\sym(A,\sigma)$ and $u\in\alt(A,\sigma)$.
Since $\alpha+w\notin\alt(Q_2,\sigma_2)$ for every $\alpha\in F$, by \cite[(3.5)]{mn} we have $\alpha+w\notin\alt(A,\sigma)$ for all $\alpha\in F$, i.e., $w\in\sym(A,\sigma)\setminus (\alt(A,\sigma)\oplus F)$, as required.
\end{proof}

The next result shows that Theorem \ref{cor} does not hold for isotropic invo\-lutions of degree $\geqslant8$ (see also Proposition \ref{deg4}).
\begin{prop}\label{count}
Let $(A,\sigma)$ be a totally decomposable algebra of degree $2^n$ with isotropic orthogonal involution over $F$.
If $n\geqslant3$ and $(A,\sigma)\not\simeq(M_{2^n}(F),t)$, then there exists an element $x\in\sym(A,\sigma)^+$ with $x^2\notin F^2$ which is not contained in any $\sigma$-invariant quaternion subalgebra of $A$.
\end{prop}

\begin{proof}
Since $\mathfrak{i}(A,\sigma)>0$, we may identify $(A,\sigma)=(B,\rho)\otimes(M_2(F),t)$, where $(B,\rho)$ is a totally decomposable algebra with orthogonal involution over $F$.
The assumptions $n\geqslant3$ and $(A,\sigma)\not\simeq(M_{2^n}(F),t)$ imply that $\deg_FB\geqslant4$ and
$(B,\rho)\not\simeq(M_{2^{n-1}}(F),t)$.
By Lemma \ref{sym}, there exists an element $w\in\sym(B,\rho)\setminus (\alt(B,\rho)\oplus F)$ and a unit $u\in\alt(B,\rho)$
for which $uw=wu$.
Set
\[x=\left(\begin{matrix}w & w+u \\w+u & w\end{matrix}\right)\in A.\]
By Remark \ref{rem} we have $x\in\sym(A,\sigma)\setminus (\alt(A,\sigma)\oplus F)$.
As $u^2\in F^\times\setminus F^{\times2}$ we have $x^2\in F^\times\setminus F^{\times2}$.
By \cite[(5.4)]{me2}, the element $x$ is not contained in any $\sigma$-invariant quaternion subalgebra of $A$, because $x+\alpha\notin\alt(A,\sigma)$ for every $\alpha\in F$.
\end{proof}

We conclude by showing that the implication $(1)\Rightarrow(2)$ in Proposition \ref{main2} and Proposition \ref{8} does not hold for arbitrary involutions.
We use the ideas of \cite[(9.4)]{dolphin3}.
Recall that the {\it canonical} involution $\gamma$ on a quaternion $F$-algebra $Q$ is defined as $\gamma(x)=x-\trd_Q(x)$ for $x\in Q$, where $\trd_Q(x)$ is the reduced trace of $x$ in $Q$.
Also, for a division algebra with involution $(D,\theta)$ over $F$ and $\alpha_1,\cdots,\alpha_n\in D^\times\cap\sym(D,\theta)$, the diagonal hermitian form $h$ on $D^n$ defined by $h(x,y)=\sum_{i=1}^n\theta(x_i)\alpha_iy_i$ is denoted by
$\langle\alpha_1,\cdots,\alpha_n\rangle_\theta$.
\begin{exm}\label{exm1}
Let $F\neq F^2$ and let $K=F(X,Y,Z)$, where $X$, $Y$ and $Z$ are indeterminates.
Let $Q=[X,Y)_K$ and let $\gamma$ be the canonical involution on $Q$.
By \cite[(9.3)]{dolphin3}, $Q$ is a division algebra over $K$.
Choose an element $s\in\sym(Q,\gamma)$ with $s^2=Y$.
Let $\psi$ be the diagonal hermitian form $\langle1,Z,s,s\rangle_\gamma$ over $(Q,\gamma)$ and set $(B,\rho)=\ad(\psi)$.
By \cite[(9.4)]{dolphin3}, $(B,\rho)$ is not totally decomposable, but $(B,\rho)_L$ is totally decomposable for every splitting field $L$ of $A$.

Now, choose $\alpha\in F^\times\setminus F^{\times2}$ and let $Q'=[X,\alpha)_K$ with a quaternion basis $(1,u,v,w)$.
Let $\tau$ be the involution on $Q'$ induced by $\tau(u)=u$ and $\tau(v)=v$.
Then $\tau$ is an orthogonal involution and $v=\tau(uv)-uv\in\alt(Q',\tau)$.
Set $(A,\sigma)=(B,\rho)\otimes_K(Q',\tau)$.
Then $(A,\sigma)$ is a central simple algebra with orthogonal involution over $K$.
We claim that $(A,\sigma)$ is totally decomposable.
Let $L=K(u)\subseteq Q'$ and set $C=C_{A}(1\otimes u)$.
Then $L/K$ is a separable quadratic extension and
\begin{align}\label{eq1}
(C,\sigma|_C)\simeq_L (B,\rho)_L,
\end{align}
 is a central simple $L$-algebra with orthogonal involution.
Since $u^2+u=X$, $Q_L\simeq[X,Y)_L$ splits, which implies that $B_L$ is also split.
It follows that $(B,\rho)_L$ is totally decomposable, i.e., $(C,\sigma|_C)$ is totally decomposable by (\ref{eq1}).
Using \cite[(7.3)]{me1} and the isomorphism (\ref{eq1}), one can find a totally decomposable algebra with orthogonal involution $(C',\sigma')$ over $K$ such that $(C,\sigma|_C)\simeq(C',\sigma')_L$.
As $C\subseteq A$, the algebra $C'$ may be identified with a subalgebra of $A$.
Let $Q''=C_{A}(C')$.
Then $Q''$ is a quaternion $K$-subalgebra of $A$ and $(A,\sigma)\simeq_K(C',\sigma')\otimes_K(Q'',\sigma|_{Q''})$ is totally decomposable, proving the claim.

The element $1\otimes v\in \alt(A,\sigma)^+$ is contained in the copy of $Q'$ in $A$, which is a $\sigma$-invariant quaternion subalgebra of $A$.
Note that $(C_A(1\otimes v),\sigma|_{C_A(1\otimes v)})\simeq (B,\rho)_{K(v)}$ as $K(v)$-algebras.
We show that $(B,\rho)_{K(v)}$ is not totally decomposable, which implies that $1\otimes v$ is not contained in any inseparable subalgebra of $(A,\sigma)$, thanks to \cite[(6.3)]{mn1}.
Since $v^2=\alpha\in F^{\times}\setminus F^{\times2}$, we have $K(v)\simeq F(\sqrt{\alpha})(X,Y,Z)$.
Hence, $Q_{K(v)}$ is still a division algebra by \cite[(9.3)]{dolphin3}.
Using \cite[(9.4)]{dolphin3} it follows that $(B,\rho)_{K(v)}$ is not totally decomposable.
\end{exm}

\noindent{\sc A.-H. Nokhodkar, {\tt
    anokhodkar@yahoo.com},\\
Department of Pure Mathematics, Faculty of Science, University of Kashan, P.~O. Box 87317-51167, Kashan, Iran.}

\end{document}